\documentclass[reqno,12pt]{amsart}
\usepackage{amssymb}
\usepackage{amsmath}
\usepackage{amsthm}
\usepackage{mathtools}
\usepackage{mathabx}
\usepackage{amsfonts}
\usepackage{xcolor}
\usepackage{graphicx}
\usepackage{tikz}
\usepackage{comment}
\usepackage{microtype}
\usepackage{enumitem}
\usepackage[
    colorlinks = true,, 
    linkcolor={blue},
	citecolor={blue},
	urlcolor={black!30!blue}
]{hyperref}
\usepackage[alphabetic, initials]{amsrefs}
\renewcommand{\MR}[1]{}

\title[Lyapunov Exponent for Gevrey Cocycles]{Continuity of Lyapunov Exponent for Quasi-Periodic Gevrey Cocycles}

\author[X.\ Wang]{Xueyin Wang}
\address{[X.\ Wang] Department of Mathematics, Texas A\&M University, College Station, TX 77843, USA}
\email{\href{mailto:xueyin@tamu.edu}{xueyin@tamu.edu}}

\theoremstyle{plain}
\newtheorem{theorem}{Theorem}[section]

\newtheorem{lemma}[theorem]{Lemma}
\newtheorem{proposition}[theorem]{Proposition}

\theoremstyle{definition}

\newtheorem{remark}[theorem]{Remark}

\numberwithin{equation}{section}

\begin{document}
\begin{abstract}
	It is shown that for the quasi-periodic cocycles in Gevrey space $G^{s}$ and subexponential Brjuno class frequency $\Omega(\eta)$, the Lyapunov exponent is continuous provided that $1<s+\eta<2$.
\end{abstract}

\maketitle	

\section{Introduction and results}
In this paper, we study the continuity of the Lyapunov exponent for one-frequency quasi-periodic $\mathrm{SL}(2,\mathbb{R})$ cocycles
\begin{equation*}
    (\alpha,A):\mathbb{T}\times \mathbb{R}^{2}\rightarrow \mathbb{T}\times \mathbb{R}^{2},\qquad (\theta, x)\mapsto (\theta+\alpha, A(\theta)x),
\end{equation*}
where $\alpha\in\mathbb{R}\setminus\mathbb{Q}$ and $A(\cdot): \mathbb{T}\to \mathrm{SL}(2,\mathbb{R})$. An important example is Schr\"odinger cocycle, where $A(\theta)=A_{E}^{V}(\theta)\coloneq\begin{pmatrix}
    E-V(\theta)&-1\\1&0
\end{pmatrix}$.

For $N\geqslant 1$, we write
\begin{equation*}
    A_{N}(\theta)=A(\theta+(N-1)\alpha)\cdots A(\theta+\alpha)A(\theta),
\end{equation*}
and define the finite-scale Lyapunov exponent by
\begin{equation*}
    L_{N}(\alpha,A)= \frac{1}{N}\int_{\mathbb{T}} \ln \|A_{N}(\theta)\| \, \mathrm{d}\theta.
\end{equation*}
By subadditivity, the Lyapunov exponent
\begin{equation*}
    L(\alpha,A)= \lim_{N\to\infty} L_{N}(\alpha,A)=\inf_{N\geqslant 1}L_{N}(\alpha,A)
\end{equation*}
exists.

The regularity of the Lyapunov exponent is a central problem in  dynamical systems and spectral theory. It is well-known that the regularity depends delicately on both the arithmetic properties of the frequency and the regularity of the cocycle \cite{youproblems}. We focus on the cocycles in Gevrey class. For $\rho>0$ and $s>1$,
let $G^{s}_{\rho}(\mathbb{T},*)$ be the space of $*$-valued functions
\begin{equation*}
    f(\theta)=\sum_{k\in\mathbb Z} \hat{f}_{k} e^{2\pi i k\theta}\quad \text{with}\quad \|f\|_{s,\rho}\coloneq \sum_{k\in\mathbb{Z}} |\hat{f}_{k}|e^{\rho|2\pi k|^{\frac{1}{s}}}<\infty.
\end{equation*}

We now introduce the frequency class used in this paper. Let $p_n/q_n$ be the continued fraction approximants of $\alpha$.
For $0<\eta<1$, we say that $\alpha\in\Omega(\eta)$ if there exists $C_{\alpha}>0$ such that for all $n$,
\begin{equation}\label{brjuno}
    \ln q_{n+1}\leqslant C_{\alpha} q_{n}^{\eta}.
\end{equation}
We call $\Omega(\eta)$ a subexponential Brjuno class. This condition is weaker than any Diophantine condition, but it is stronger than the usual Brjuno condition.

The main result of this paper is the following.

\begin{theorem}\label{mainthm}
    Fix $\rho>0$, $1<s<2$, $0<\eta<2-s$ and $\alpha\in \Omega(\eta)$. Then the map 
    \begin{equation*}
        A\mapsto L(\alpha,A)
    \end{equation*}
    is continuous in $G^{s}_{\rho}(\mathbb{T},\mathrm{SL}(2,\mathbb{R}))$.
\end{theorem}

Let us compare our result with some known results. In the analytic setting, \cite{MR1933451} proved the continuity of the Lyapunov exponent for one-frequency quasi-periodic Schr\"odinger cocycles for arbitrary irrational frequencies. This result was later extended to general $\mathrm{SL}(2,\mathbb{R})$ cocycles in \cite{MR2563096}. In contrast, in the non-analytic setting, a surprising result due to \cite{MR3127804} shows that the Lyapunov exponent may be discontinuous in the $C^{\infty}$ topology. This dichotomy naturally motivates the study of the regularity of the Lyapunov exponent for cocycles in Gevrey classes \cites{MR2108112,MR3291922,MR4477429,powell2025large, MR4811662}, which provide a natural intermediate regularity between the analytic and $C^{\infty}$ topologies.

The regularity index $s$ of the Gevrey class $G^{s}$ plays a key role in the study of the Lyapunov exponent. For $1<s<2$, \cite{MR2108112} first proved that the Lyapunov exponent is continuous for Schr\"odinger cocycles with strong Diophantine ($SDC$) frequencies. Here, $\alpha$ is said to be $SDC$ if there exist $C_{\alpha}>0$ and $\tau>1$ such that
\begin{equation*}
    q_{n+1}\leqslant C_{\alpha} q_{n}(\ln q_{n})^{\tau}.
\end{equation*}
Later, \cite{MR4477429} obtained corresponding continuity results for $\mathrm{SL}(2,\mathbb{R})$ cocycles and Diophantine ($DC$) frequencies, which satisfy
\begin{equation*}
    q_{n+1}\leqslant C_{\alpha}q_{n}^{\tau}
\end{equation*}
for some $C_{\alpha}>0$ and $\tau>1$. Since $SDC\subseteq DC\subseteq \Omega(\eta)$ for any $\eta>0$, our Theorem \ref{mainthm} covers all frequencies considered in \cites{MR2108112,MR4477429} when $1<s<2$.

In contrast, one cannot expect continuity for Gevrey cocycles when $s>2$. The first counterexample was constructed in \cite{ge2021transition}. More precisely, they proved that for any bounded type frequency, there exists a Schr\"odinger cocycle whose corresponding Lyapunov exponent is discontinuous. Here, $\alpha$ is called bounded type if there is $C_{\alpha}>0$ such that
\begin{equation*}
    q_{n+1}\leqslant C_{\alpha}q_{n}.
\end{equation*}
Recently, \cite{liang2025discontinuity} further constructed discontinuity examples for general $\mathrm{SL}(2,\mathbb{R})$ cocycles covering all $\alpha\in SDC$. Since the $SDC$ class has full Lebesgue measure, the result in \cite{liang2025discontinuity} reveals that $s=2$ is a critical transition point of continuity for almost all frequencies.

For the Schr\"odinger cocycle, we denote $L(\alpha,E)\coloneq L(\alpha,A_{E}^{V})$. Then we have the following corollary.
\begin{theorem}\label{corosc}
    Fix $\rho>0$, $1<s<2$, $0<\eta<2-s$, $\alpha\in \Omega(\eta)$, and $V\in G^{s}_{\rho}(\mathbb{T},\mathbb{R})$. Then $L(\alpha,E)$ is continuous with respect to $E$.
\end{theorem}

Both Theorem \ref{mainthm} and Theorem \ref{corosc} imply the competition between the smoothness of the cocycle and the arithmetic of the frequency. A similar phenomenon was also observed by \cite{liang2025discontinuity}, where the authors demonstrated that lower smoothness combined with rapid denominator growth can result in discontinuity (see Theorems 1.3--1.5 in \cite{liang2025discontinuity}).
Our Theorem \ref{mainthm} and Theorem \ref{corosc} quantitatively capture this competition, implying that the larger $s$ is (i.e., the lower the smoothness), the smaller $\eta$ must be (i.e., the slower the denominator growth) to ensure the continuity of the Lyapunov exponent.

Finally, we briefly describe the idea of the proof. The starting point is the large deviation theorem for Gevrey cocycles, which is available in a scale window $C_1q_{s}^\sigma<N<C_2q_{s}^{\sigma_1}$. In the Diophantine case, the continued fraction denominators grow at most polynomially. To connect these scale windows, the multi-scale induction developed in \cite{MR4477429} relies on choosing consecutive scales with a subexponential growth ratio, specifically $N_{s+1}/N_{s} > \exp(cq_{s}^{c})$ (see \cite{MR4477429}*{Eq.~(8.30)}). This rapid growth subsequently yields a subexponentially decaying error bound between the Lyapunov exponent and its finite-scale approximation. However, for $\alpha\in\Omega(\eta)$, such subexponential estimates are invalid because the gaps between consecutive denominators can be significantly larger.

The novelty of our approach lies in a refined multi-scale induction scheme that successfully bridges these larger gaps. Our key observation is that one can delicately choose the parameters in the Gevrey large deviation theorem so that the Brjuno condition \eqref{brjuno} still allows the scale windows to be connected. Instead of forcing a subexponential bound, we establish a polynomial growth bound for the successive scale ratios: $N_{s+1}/N_{s} > q_{s}^{c}$. Consequently, as demonstrated in Theorem \ref{approx}, this slower growth leads to an approximation error with polynomial decay, rather than subexponential decay. Nevertheless, this polynomial error remains sufficient to establish the desired continuity. This is because our final argument ultimately only requires that $L(\alpha,A)$ be approximated by a fixed finite-scale expression, $2L_{2N_0}(\alpha,A)-L_{N_0}(\alpha,A)$, with an error tending to zero.


\section{Large deviation theorem}

Let us recall the following large deviation theorem. For simplicity, we suppress the explicit dependence on $\alpha$ in the $L_{N}(\alpha,A)$ and $L(\alpha,A)$ since $\alpha\in\Omega(\eta)$ is fixed.  In the notation of \cite{MR4477429}, our $G^{s}_{\rho}$ corresponds to $G^{\nu}_{\rho}$ with $\nu=1/s\in (1/2,1)$.
\begin{theorem}[\cite{MR4477429}]\label{ldt}
    Let $\rho>0$, $1<s<2$, $0<\kappa<1$. Assume $A\in G^{s}_{\rho}(\mathbb{T},\mathrm{SL}(2,\mathbb{R}))$, and 
    \begin{equation*}
        \bigg|\alpha-\frac{a}{q} \bigg|<\frac{1}{q^{2}},  \quad \gcd(a,q)=1.
    \end{equation*}
    Then there exist $c, C_{i}(\kappa)>0, i=1,2$, $\sigma_{1}>\sigma>1>\gamma>0$ and $q_{0}(\kappa,\rho,s, \|A\|_{s,\rho})>0$ such that for $q\geqslant q_{0}$ and $C_{1}(\kappa) q^{\sigma}<N<C_{2}(\kappa)q^{\sigma_{1}}$,
    \begin{equation*}
        \operatorname{meas}\bigg\{\theta\in\mathbb{T}: \bigg|\frac{1}{N}\ln \|A_{N}(\theta)\|-L_{N}(A)\bigg|>\kappa \bigg\}<e^{-cq^{\gamma}}.
    \end{equation*}
\end{theorem}
\begin{remark}
    The constants are uniform in a neighborhood of $A$ in $G^{s}_{\rho}$.
\end{remark}
\begin{remark}\label{paraold}
    In \cite{MR4477429}*{Sect 8.1.3}, the parameters $\sigma_{1},\sigma,\gamma$ in Theorem \ref{ldt} are chosen as follows: Let 
    \begin{equation}\label{dold}
        \delta\in (s-1,1).
    \end{equation}
    Let $\sigma>1$ and $p\in\mathbb{N}$ such that
    \begin{equation}\label{p1}
        1<\sigma<\frac{1}{\delta}\quad \text{and}\quad \frac{\delta\sigma}{\sigma-1}<p<\frac{1}{\sigma-1}.
    \end{equation}
    Define
    \begin{equation}\label{p2}
        \gamma=1+p(1-\sigma)\quad \text{and}\quad \sigma_{1}=\frac{p(\sigma-1)}{\delta}.
    \end{equation}
\end{remark}

The parameter choices detailed in Remark \ref{paraold} are designed for the Avalanche Principle in the Diophantine setting. Relaxing this frequency condition to the subexponential Brjuno class \eqref{brjuno} requires confining the parameters to a strictly tighter regime. More precisely, we replace \eqref{dold} with (because $s+\eta<2$)
\begin{equation*}
    \delta\in (s-1,1-\eta).
\end{equation*}
Then $\sigma,p,\gamma,\sigma_{1}$ are chosen as in \eqref{p1} and \eqref{p2}. 

\begin{lemma}\label{paranew}
    For any $\delta \in (s-1, 1-\eta)$, there exist parameters $\sigma, p, \gamma,$ and $\sigma_{1}$ satisfying \eqref{p1}, \eqref{p2}, and
    \begin{equation}\label{gap}
        \eta \sigma_{1} < \gamma\sigma.
    \end{equation}
\end{lemma}

\begin{proof}
    For $\sigma>1$, we choose the integer
    \begin{equation*}
        p=\bigg\lfloor \frac{\delta\sigma}{\sigma-1} \bigg\rfloor+1.
    \end{equation*}
    Note that as  $\sigma \to 1$, we have $p(\sigma-1)<1$ and $p(\sigma-1)\to \delta$. Consequently, it follows that
    \begin{equation*}
        \frac{\sigma_{1}}{\sigma} = \frac{p(\sigma-1)}{\delta \sigma} \to 1 \quad \text{and} \quad \gamma = 1+p(1-\sigma) \to 1-\delta.
    \end{equation*}
    Since $\eta < 1-\delta$, it follows that \eqref{gap} holds provided $\sigma$ is chosen sufficiently close to $1$.
\end{proof}

\section{Applications of avalanche principle}
In this section, we establish the error estimate on two scales.
\begin{theorem}\label{map}
    Suppose that $|\alpha - a/q| < 1/q^{2}$ with $\gcd(a,q) = 1$. Let $q \geqslant q_{0}$ and $N$ satisfying $C_{1}(\kappa)q^{\sigma} < N <2N< C_{2}(\kappa)q^{\sigma_{1}}$ be as in Theorem \ref{ldt}. Further assume that $L_{N}(A) > 90\kappa > 0$ and $L_{2N}(A) > \frac{9}{10}L_{N}(A)$. Then, for sufficiently large $q_{0}$ and for any $N'$ such that $N \mid N'$ and $m \coloneq N'/N < e^{\frac{c}{10}q^{\gamma}}$, we have
    \begin{equation*}
        |L_{N'}(A) + L_{N}(A) - 2L_{2N}(A)| < e^{-\frac{c}{2}q^{\gamma}} + \frac{2L_{N}(A)}{m}.
    \end{equation*}
\end{theorem}

The proof is based on the following  avalanche principle.
\begin{proposition}[\cite{MR1847592}]\label{AP}
    Let $M_{1},\cdots, M_{n}$ be a sequence in $\mathrm{SL}(2,\mathbb{R})$ satisfying
    \begin{equation*}
        \|M_{j}\|\geqslant \mu>n, \quad j=1\cdots,n,
    \end{equation*}
    and for $j=1,\cdots, n-1$,
    \begin{equation*}
        |\ln \|M_{j}\|+\ln \|M_{j+1}\| -\ln  \|M_{j+1}M_{j}\||< \frac{1}{2}\ln \mu.
    \end{equation*}
    Then
    \begin{equation*}
        \bigg|\ln \|\prod_{j=1}^{n} M_{j}\|+ \sum_{j=2}^{n-1} \ln \|M_{j}\|-\sum_{j=1}^{n-1} \ln \|M_{j+1}M_{j}\|\bigg|<C\frac{n}{\mu}.
    \end{equation*}
\end{proposition}

\begin{proof}[Proof of Theorem \ref{map}]
    For brevity, we write $L_{N} \coloneq L_{N}(A)$. We apply Proposition \ref{AP} to the matrices $M_{j}(\theta) \coloneq A_{N}(\theta + jN\alpha)$, $j=0,\cdots, m-1$. By Theorem \ref{ldt}, there exists a subset $\Omega_{N} \subseteq \mathbb{T}$ satisfying
    \begin{equation*}
        \operatorname{meas}(\mathbb{T} \setminus \Omega_{N}) < 2m e^{-cq^{\gamma}},
    \end{equation*}
    such that for all $\theta \in \Omega_{N}$,
    \begin{equation*}
        \begin{split}
            |\ln \|M_{j}\| - N L_{N}| &< N \kappa, \\
            |\ln \|M_{j+1}M_{j}\| - 2N L_{2N}| &< 2N\kappa.
        \end{split}
    \end{equation*}
    Observe that for $\theta \in \Omega_{N}$, we have
    \begin{equation*}
        \|M_{j}\| > e^{N(L_{N} - \kappa)} > e^{\frac{89}{90} N L_{N}}
    \end{equation*}
    and
    \begin{equation*}
        \begin{split}
            \big| \ln \|M_{j}\| + \ln \|M_{j+1}\| &- \ln \|M_{j+1}M_{j}\| \big| \\
            &< 4N\kappa + 2N |L_{N} - L_{2N}| < \frac{1}{4} N L_{N}.
        \end{split}
    \end{equation*}
    Apply Proposition \ref{AP} with $\mu=e^{\frac{89}{90} N L_{N}}$, one can see that $\mu>e^{\frac{c}{10}q^{\gamma}}>m$ and $\frac{1}{4}NL_{N}<\frac{1}{2}\ln \mu$ due to the largeness of $q_{0}$, and thus for $\theta \in \Omega_{N}$,
    \begin{equation*}
        \bigg| \ln \bigg\| \prod_{j=0}^{m-1} M_{j} \bigg\| + \sum_{j=1}^{m-2} \ln \|M_{j}\| - \sum_{j=0}^{m-2} \ln \|M_{j+1}M_{j}\| \bigg| < C m e^{-\frac{89}{90}NL_{N}}.
    \end{equation*}
    Integrating over $\Omega_{N}$, we obtain
    \begin{equation*}
        \begin{split}
            \bigg| \int_{\Omega_{N}} \ln \|A_{N'}(\theta)\| \,\mathrm{d}\theta &+ \sum_{j=1}^{m-2} \int_{\Omega_{N}} \ln \|A_{N}(\theta + jN\alpha)\| \,\mathrm{d}\theta \\ 
            &- \sum_{j=0}^{m-2} \int_{\Omega_{N}} \ln \|A_{2N}(\theta + jN\alpha)\| \,\mathrm{d}\theta \bigg| < C m e^{-\frac{89}{90}NL_{N}}.
        \end{split}
    \end{equation*}
    Combining this estimate with the integral over $\mathbb{T} \setminus \Omega_{N}$, we find
    \begin{equation*}
        \bigg| L_{N'} + \frac{m-2}{m} L_{N} - 2\frac{m-1}{m} L_{2N} \bigg| < C \frac{m}{N'} e^{-\frac{89}{90}NL_{N}} +  2\|A\|_{C^{0}}m e^{-cq^{\gamma}} < e^{-\frac{c}{2}q^{\gamma}}.
    \end{equation*}
    It follows from the subadditivity of $L_{N}$ that
    \begin{equation*}
        |L_{N'} + L_{N} - 2L_{2N}| < e^{-\frac{c}{2}q^{\gamma}} + \frac{2L_{N}}{m}.
    \end{equation*}
\end{proof}

\section{Multi-scale induction}
We establish the local uniform error estimate between Lyapunov exponent and its finite-scale approximation.
\begin{theorem}\label{approx}
    Assume that $\alpha\in \Omega(\eta)$ and $L(A)>100\kappa>0$. There exist $0<\varepsilon<1$, $c'>0$, and $C_{1}\tilde{q}_{0}^{\sigma}<N_{0}<C_{2}\tilde{q}_{0}^{\sigma_{1}}$ such that for sufficiently large $\tilde{q}_{0}$ and any $B$ satisfying $\|A-B\|_{s,\rho}<\varepsilon$, we have
    \begin{equation*}
        |L(B)+L_{N_{0}}(B)-2L_{2N_{0}}(B)|< \tilde{q}_{0}^{-c'}.
    \end{equation*}
\end{theorem}

The proof of Theorem \ref{approx} utilizes multi-scale induction. We fix any $\alpha\in \Omega(\eta)$. 
Let $\{\tilde{a}_{i}/\tilde{q}_{i}\}$ be a subsequence of the continued fraction expansion of $\alpha$. We inductively choose the sequences $\{\tilde{q}_{s}\}_{s\geqslant 0}$ and $\{N_{s}\}_{s\geqslant
 0}$. Let's start from $s=0$. 

\begin{lemma}[\cite{MR4477429}]\label{first}
    Assume that $L(A)>100\kappa>0$. There exists $N_{0}$ with $C_{1}(\kappa)q_{0}^{\sigma}<N_{0}<2N_{0}<C_{2}(\kappa)q_{0}^{\sigma_{1}}$ such that
    \begin{equation*}
        L_{2N_{0}}(A)>\frac{99}{100}L_{N_{0}}(A).
    \end{equation*}
\end{lemma}

Now we construct the sequences for $s\geqslant 1$. Recall that
$\frac{\sigma_{1}}{\sigma}< \frac{\gamma}{\eta}$ by Lemma \ref{paranew}. Fix any number $\zeta$ satisfying
\begin{equation}\label{zeta}
    \frac{\sigma_{1}}{\sigma}<\zeta< \frac{\gamma}{\eta}.
\end{equation}

\begin{lemma}\label{induction}
    There exist $\{\tilde{q}_{s}\}_{s\geqslant 0}$ and $\{N_{s}\}_{s\geqslant 0}$ such that for all $s\geqslant 0$:
\begin{equation}\label{qs}
    \tilde{q}_{s+1} \ \text{is the smallest}\ q_{j}\ \text{such that}\ \tilde{q}_{s+1}> \tilde{q}_{s}^{\zeta}, 
\end{equation}
\begin{equation}\label{Ns}
    C_{1}(\kappa)\tilde{q}_{s}^{\sigma}<N_{s}<2N_{s}<C_{2}(\kappa) \tilde{q}_{s}^{\sigma_{1}},
\end{equation}
\begin{equation}\label{ms}
    N_{s+1}=m_{s+1}N_{s},\ \tilde{q}_{s}^{\zeta \sigma-\sigma_{1}}< m_{s+1}<2m_{s+1}<e^{\frac{c}{10}\tilde{q}_{s}^{\gamma}}.
\end{equation}
\end{lemma}
\begin{proof}
We choose $\tilde{q}_{0}=q_{0}$ sufficiently large and define $\tilde{q}_{s}$ as in \eqref{qs}. We use induction to prove the existence of $N_{s}$ satisfying \eqref{Ns} and \eqref{ms} for any $s\geqslant 0$.

For $s=0$, the existence of $N_{0}$ is ensured by Lemma \ref{first}.

For $s\geqslant 0$, we assume that $N_{s}$ has been constructed so that \eqref{Ns} holds. We define $m_{s+1}$ and $N_{s+1}$ as follows.
Let $q_{j-1}$ be the denominator before $q_{j}=\tilde{q}_{s+1}$. Then we have $q_{j-1}\leqslant \tilde{q}_{s}^{\zeta}$. By \eqref{brjuno}, we have
\begin{equation*}
    \ln \tilde{q}_{s+1} =\ln q_{j}\leqslant C_{\alpha} q_{j-1}^{\eta} \leqslant C_{\alpha}\tilde{q}_{s}^{\eta\zeta}.
\end{equation*}
Thus, by the selection of $\tilde{q}_{s}$ and \eqref{zeta}, one has
\begin{equation}\label{qs+1}
    \tilde{q}_{s}^{\zeta}<\tilde{q}_{s+1} \leqslant  \exp(C_{\alpha}\tilde{q}_{s}^{\zeta \eta}) < e^{\frac{c}{10\sigma}\tilde{q}_{s}^{\gamma}},
\end{equation}
where $c$ is defined in Theorem \ref{ldt} and $\tilde{q}_{0}$ is sufficiently large.
Take $N_{s+1}=m_{s+1}N_{s}$ with \begin{equation*}
    m_{s+1}=\bigg\lfloor\frac{ (C_{1}+C_{2})\tilde{q}_{s+1}^{\sigma}}{N_{s}}\bigg\rfloor+1.
\end{equation*}
One can check that (by \eqref{zeta} and $\sigma_{1}>\sigma$), for $\tilde{q}_{0}$ sufficiently large,
\begin{equation*}
    \begin{split}
        &N_{s+1}\geqslant  (C_{1}+C_{2})\tilde{q}_{s+1}^{\sigma}\geqslant C_{1}\tilde{q}_{s+1}^{\sigma}, \\
        &2N_{s+1}\leqslant 4 (C_{1}+C_{2})\tilde{q}_{s+1}^{\sigma} \leqslant C_{2}\tilde{q}_{s+1}^{\sigma_{1}},
    \end{split}
\end{equation*}
which proves \eqref{Ns} for $N_{s+1}$. By \eqref{Ns} and \eqref{qs+1},
\begin{equation*}
    \begin{split}
        &m_{s+1}\geqslant  (C_{1}+C_{2})\tilde{q}_{s+1}^{\sigma} C_{2}^{-1}\tilde{q}_{s}^{-\sigma_{1}} >\tilde{q}_{s}^{\sigma \zeta-\sigma_{1}},\\
        &2m_{s+1}\leqslant 4C_{1}^{-1}\tilde{q}_{s}^{-\sigma}(C_{1}+C_{2})\tilde{q}_{s+1}^{\sigma} \leqslant \tilde{q}_{s+1}^{\sigma} \leqslant  e^{\frac{c}{10}\tilde{q}_{s}^{\gamma}}.
    \end{split}
\end{equation*}
This proves \eqref{ms} for $N_{s+1}$.
\end{proof}

Now we finish the proof of Theorem \ref{approx}.
\begin{proof}[Proof of Theorem \ref{approx}]
Let $C_{0}\coloneq 10(\|A\|_{C^{0}}+1)$. For brevity, we drop the dependency on $B$ and write $L_{N}\coloneq L_{N}(B)$ when there is no ambiguity.
We proceed by induction to show that the sequences $\{\tilde{q}_{s}\}$ and $\{N_{s}\}$ defined in Lemma \ref{induction} additionally satisfy the following bounds for any $s\geqslant 0$:
\begin{equation}\label{Q1}
    |L_{N_{s+1}}+L_{N_{s}}-2L_{2N_{s}}|<C_{0}\tilde{q}_{s}^{\sigma_{1}-\zeta\sigma},
\end{equation}
\begin{equation}\label{Q2}
    |L_{2N_{s+1}}-L_{N_{s+1}}|<2C_{0}\tilde{q}_{s}^{\sigma_{1}-\zeta\sigma},
\end{equation}
\begin{equation}\label{Q3}
    |L_{N_{s+1}}-L_{N_{s}}|< 10C_{0}\tilde{q}_{s-1}^{\sigma_{1}-\zeta\sigma},
\end{equation}
where we set $\tilde{q}_{-1}=1$ for simplicity.

{\bf Base case: }For the base case $s=0$, the assumption $L(A)>100\kappa>0$ and the subadditivity of the Lyapunov exponent imply $L_{N_{0}}(A)>100\kappa>0$. By Lemma \ref{first}, we have $L_{2N_{0}}(A)>\frac{99}{100}L_{N_{0}}(A)$. Since the finite-scale Lyapunov exponent $L_{N_{0}}(\cdot)$ is continuous, there exists $0<\varepsilon<1$ (depending on $A$, $N_{0}$, and $\kappa$) such that for any $\|A-B\|_{s,\rho}<\varepsilon$,
\begin{equation}\label{aB}
    L_{N_{0}}(B)> 99\kappa \quad \text{and} \quad L_{2N_{0}}(B)>\frac{98}{100}L_{N_{0}}(B).
\end{equation}

Applying Theorem \ref{map} with $N'=N_{1}$, $m=m_{1}$, and $N=N_{0}$, and invoking Lemma \ref{induction}, we obtain
\begin{equation}\label{N1}
    |L_{N_{1}}+L_{N_{0}}-2L_{2N_{0}}|<e^{-\frac{c}{2}\tilde{q}_{0}^{\gamma}} + \frac{2L_{N_{0}}}{m_{1}}< C_{0}\tilde{q}_{0}^{\sigma_{1}-\zeta\sigma}.
\end{equation}
Applying Theorem \ref{map} again with $N'=2N_{1}$, $m=2m_{1}$, and $N=N_{0}$ yields
\begin{equation}\label{2N1}
    |L_{2N_{1}}+L_{N_{0}}-2L_{2N_{0}}|<C_{0}\tilde{q}_{0}^{\sigma_{1}-\zeta\sigma}.
\end{equation}
Combining \eqref{N1} and \eqref{2N1}, we deduce
\begin{equation*}
    |L_{2N_{1}}-L_{N_{1}}|<2C_{0}\tilde{q}_{0}^{\sigma_{1}-\zeta\sigma}.
\end{equation*}
By \eqref{aB} and \eqref{N1},
\begin{equation}\label{N1N0}
    \begin{split}
        |L_{N_{1}}-L_{N_{0}}|&\leqslant |L_{N_{1}}+L_{N_{0}}-2L_{2N_{0}}|+ 2|L_{2N_{0}}-L_{N_{0}}|\\
        &<  C_{0}\tilde{q}_{0}^{\sigma_{1}-\zeta\sigma} +\frac{4}{100} L_{N_{0}}.
    \end{split}
\end{equation}
Since $\tilde{q}_{-1}=1$, we trivially have
\begin{equation*}
    |L_{N_{1}}-L_{N_{0}}|<10 C_{0}\tilde{q}_{-1}^{\sigma_{1}-\zeta\sigma}.
\end{equation*}
Hence, \eqref{Q1}, \eqref{Q2}, and \eqref{Q3} hold for $s=0$.

{\bf Inductive case: } Now, assume that \eqref{Q1}, \eqref{Q2}, and \eqref{Q3} hold for all $s\leqslant j-1$.
Note that by taking $\tilde{q}_{0}$ sufficiently large, \eqref{N1N0} implies
\begin{equation}\label{N1lower}
    L_{N_{1}}>\frac{48}{50} L_{N_{0}}-C_{0}\tilde{q}_{0}^{\sigma_{1}-\zeta\sigma}>95\kappa.
\end{equation}
Combining \eqref{N1lower} with the inductive hypothesis \eqref{Q3}, for sufficiently large $\tilde{q}_0$, we have
\begin{equation*}
    \begin{split}
        L_{N_{j}}&\geqslant L_{N_{1}}-\sum_{s=1}^{j-1}|L_{N_{s+1}}-L_{N_{s}}|\\
        &\geqslant 95\kappa- \sum_{s=1}^{j-1} 10 C_{0}\tilde{q}_{s-1}^{\sigma_{1}-\zeta\sigma} > 90\kappa.
    \end{split}
\end{equation*}
By the inductive hypothesis \eqref{Q2}, we have 
\begin{equation*}
    |L_{2N_{j}}-L_{N_{j}}|< 2 C_{0}\tilde{q}_{j-1}^{\sigma_{1}-\zeta\sigma}.
\end{equation*}
For sufficiently large $\tilde{q}_0$, this implies 
\begin{equation*}
    L_{2N_{j}}> L_{N_{j}} -2 C_{0}\tilde{q}_{j-1}^{\sigma_{1}-\zeta\sigma}>L_{N_{j}} - 9\kappa> \frac{9}{10}L_{N_{j}}.
\end{equation*}

Applying Theorem \ref{map} with $N'=N_{j+1}$, $m=m_{j+1}$, $N=N_{j}$, and $q=\tilde{q}_{j}$, we get
\begin{equation}\label{Nj}
    |L_{N_{j+1}}+L_{N_{j}}-2L_{2N_{j}}|<e^{-\frac{c}{2}\tilde{q}_{j}^{\gamma}} + \frac{2L_{N_{j}}}{m_{j+1}}<C_{0}\tilde{q}_{j}^{\sigma_{1}-\zeta\sigma}.
\end{equation}
Applying Theorem \ref{map} once more with $N'=2N_{j+1}$, $m=2m_{j+1}$, $N=N_{j}$, and $q=\tilde{q}_{j}$, we find
\begin{equation}\label{2Nj}
    |L_{2N_{j+1}}+L_{N_{j}}-2L_{2N_{j}}|<C_{0}\tilde{q}_{j}^{\sigma_{1}-\zeta\sigma}.
\end{equation}
Combining \eqref{Nj} and \eqref{2Nj} yields
\begin{equation*}
    |L_{2N_{j+1}}-L_{N_{j+1}}|<2C_{0}\tilde{q}_{j}^{\sigma_{1}-\zeta\sigma}.
\end{equation*}
Furthermore, the triangle inequality gives
\begin{equation*}
    \begin{split}
        |L_{N_{j+1}}-L_{N_{j}}|&\leqslant |L_{N_{j+1}}+L_{N_{j}}-2L_{2N_{j}}|+ 2|L_{2N_{j}}-L_{N_{j}}|< 10 C_{0}\tilde{q}_{j-1}^{\sigma_{1}-\zeta\sigma}.
    \end{split}
\end{equation*}
Thus, \eqref{Q1}, \eqref{Q2}, and \eqref{Q3} hold for $s=j$.

Consequently, utilizing the telescopic sum $L - L_{N_1} = \sum_{s\geqslant 1}(L_{N_{s+1}}-L_{N_s})$ and the growth condition $\tilde{q}_{s+1}>\tilde{q}_{s}^{\zeta}$ with $\zeta>1$, we deduce
\begin{equation*}
\begin{split}
    |L+L_{N_{0}}-2L_{2N_{0}}|&\leqslant |L_{N_{1}}+L_{N_{0}}-2L_{2N_{0}}| + \sum_{s\geqslant 1} |L_{N_{s+1}}-L_{N_{s}}|\\
    &\leqslant C_{0}\tilde{q}_{0}^{\sigma_{1}-\zeta\sigma} + 10C_{0}\sum_{s\geqslant 1}\tilde{q}_{s-1}^{\sigma_{1}-\zeta\sigma} \leqslant \tilde{q}_{0}^{-c'},
\end{split}
\end{equation*}
where $c'\coloneq \frac{1}{2}(\zeta\sigma-\sigma_{1})>0$, and the final inequality holds for sufficiently large $\tilde{q}_{0}$. This completes the proof.
\end{proof}

\subsection{Proof of Theorem \ref{mainthm}}
Assume that $A, A_{n}\in G^{s}_{\rho}(\mathbb{T},\mathrm{SL}(2,\mathbb{R}))$ and $A_{n}\to A$ as $n\to \infty$. If $L(A)=0$, then the continuity of $L(\cdot)$ at $A$ follows directly from upper semi-continuity:
\begin{equation*}
    0\leqslant \liminf_{n\to\infty} L(A_{n})\leqslant \limsup_{n\to\infty} L(A_{n})\leqslant L(A)=0.
\end{equation*}
Therefore, we may assume $L(A)>100\kappa>0$. 

By Theorem \ref{approx}, for sufficiently large $n$, we simultaneously have
\begin{equation*}
    \begin{split}
        |L(A)+L_{N_{0}}(A)-2L_{2N_{0}}(A)|&<\tilde{q}_{0}^{-c'},\\
        |L(A_{n})+L_{N_{0}}(A_{n})-2L_{2N_{0}}(A_{n})|&<\tilde{q}_{0}^{-c'}.
    \end{split}
\end{equation*}
It follows that
\begin{equation*}
    \begin{split}
        |L(A)-L(A_{n})|&\leqslant |L_{N_{0}}(A)-L_{N_{0}}(A_{n})| +2|L_{2N_{0}}(A)-L_{2N_{0}}(A_{n})|+2\tilde{q}_{0}^{-c'}\\
        &\leqslant C(N_{0}) \|A_{n}-A\|_{C^{0}}+2\tilde{q}_{0}^{-c'}.
    \end{split}
\end{equation*}
Consequently, 
\begin{equation*}
    \limsup_{n\to\infty}|L(A)-L(A_{n})| \leqslant 2\tilde{q}_{0}^{-c'}.
\end{equation*}
Letting $\tilde{q}_{0}\to\infty$ implies the result.
\qed

\section*{Acknowledgments}
X. Wang thanks Prof. Wencai Liu for useful discussion. This work was supported by NSF DMS-2246031
and NSF DMS-2052572.




\bibliography{main}
\end{document}